\documentclass[a4paper]{amsproc}
\usepackage{amssymb}
\usepackage{amscd} 
\usepackage[dvips]{graphicx} 

\theoremstyle{plain}

\theoremstyle{definition}

\theoremstyle{remark}
  
 \numberwithin{equation}{section}

\renewcommand{\leq}{\leqslant}
\renewcommand{\geq}{\geqslant}

\newcommand{\field}[1]{\mathbb{#1}}
\newcommand{\RR}{\field{R}}

\newcommand{\dn}{{\text{\tiny dn}}}

\newcommand{\up}{{\text{\tiny up}}}

\newcommand{\zl}{\lambda}

\newcommand{\zO}{\Omega}
\newcommand{\ftran}[1]{\mathcal{F}\left\{{#1}\right\}}

\newcommand{\iftran}[1]{\mathcal{F}^{-1}\left\{{#1}\right\}}

\theoremstyle{plain}\newtheorem{theorem}{Theorem}[section]
\theoremstyle{remark}\newtheorem{remark}[theorem]{Remark}
\theoremstyle{plain}\newtheorem{lemma}[theorem]{Lemma}
\theoremstyle{plain}
\theoremstyle{definition}
\theoremstyle{example}

\title[Emerging problems in approximation theory]{Emerging problems in approximation theory for the numerical solution of  nonlinear PDEs of integrable type}


\author[Fermo]{\bfseries L. Fermo}

\address{
Department of Mathematics and Computer Science\\ 
University of Cagliari   \\ 
Cagliari\\
Italy}
\email{fermo@unica.it,cornelis@krein.unica.it,seatzu@unica.it}

\author[Van der Mee]{C. Van der Mee}

\author[Seatzu]{S. Seatzu}

\thanks{The research was partially supported by INDAM and by Autonomous Region of Sardinia under grant L.R.7/2007 ``Promozione della Regione Scientifica e della Innovazione Tecnologica in Sardegna''.} 

\dedicatory{Communicated by }

\begin{document}

\vspace{18mm}
\setcounter{page}{1}
\thispagestyle{empty}

\begin{abstract}
In this paper we present some open problems pertaining to the approximation theory involved in the solution of the important class of Nonlinear Partial Differential Equations (NPDEs) of integrable type. For this class of NPDEs, any Initial Value Problem (IVP) can be theoretically solved by the Inverse Scattering Transform (IST) technique whose main steps involve the solution of Volterra equations with structured kernels on unbounded domains, the solution of  Fredholm integral equations and the identification of coefficients and parameters of monomial-exponential sums. The aim of this paper is twofold: propose a method for solving the above mentioned problems under particular hypothesis and arouse interest in these problems in order to develop an effective method which works under more general assumptions. 
\end{abstract}

\maketitle

\section{Introduction}  

The class of Nonlinear Partial Differential Equations (NPDEs) of integrable type is important in mathematics as in several applicative areas of  physics, biology and engineering \cite{AC} \cite{Ablowitz2004}, \cite{AS}, \cite{FT}, \cite{NMPZ}  .
For this special class of NPDEs,  the nonlinear Schr\"odinger  (NLS) equation, which arises in modeling  electromagnetic waves in optical fibers as well as  waves on the surface of deep water, has a special role in mathematics \cite{Agrawal}, \cite{Mn74}, \cite{Shaw}, \cite{ZS}.
Firstly, we recall that the NLS equation is expressed as
\begin{equation}\label{NLS}
{\bf{i}}u_t+u_{xx}\pm 2 |u|^2u=0, \quad u=u(x,t), \quad x\in \mathbb{R}, \quad t>0 
\end{equation}
where ${\bf{i}}$ denotes the imaginary unit, the subscripts $x$ and $t$ denote the partial derivatives with respect to position $x$ and time $t$ and the sign $\pm$  depends on the symmetry properties of the model we are addressing. In particular, the plus sign appears in the focusing case and the minus sign in the defocusing case, which represent the two most important situations. 

We are interested in the initial value problem (IVP) for the NLS, that is in considering \eqref{NLS}, given the initial solution     
				$$ u_0(x)=u(0,x), \quad x \in \mathbb{R}, \quad u_0 \in L^1(\RR).$$
Following the path of the IST \cite{Ablowitz2004}, its solution can be obtained by solving, in order,  the following three problems:
\begin{itemize}
\item[(a)] determine the initial scattering data, given its initial solution;
\item[(b)] propagate the initial scattering data in time;
\item[(c)] solve two systems of integral equations whose kernels codify  the initial scattering data evolved in time.
\end{itemize}

From the numerical point of view, the problem of most interest is the first one, as effective methods to solve the two other problems have been developed recently, under the assumption that the initial scattering data are  known \cite{ArRoSe2011}.

Let us now illustrate the organization of the paper. In Section \ref{ZSsystem} we discuss the Zakharov-Shabat system, which gives a complete characterization of the scattering data associated to the NLS we want to compute, that is the transmission coefficient, the reflection coefficients (from the left and from the right), the bound states and the norming constants. Section \ref{auxiliaryfunctions} is devoted to the introduction and characterization of the auxiliary functions whose approximation is basic to evaluating all of the mentioned scattering data. In Section \ref{Marchenko} we introduce the initial Marchenko kernels, which codify the scattering data and that can be computed by solving Volterra integral equations. In Section \ref{computational} we propose a numerical method to compute the scattering data in the reflectionless case. In Section \ref{tests} we present the numerical results which confirm the effectiveness of the method in this particular case. Section \ref{conclusion} is devoted to conclusions and perspectives.

\section{Initial scattering data}\label{ZSsystem}
The characterization of the initial scattering data is based on the spectral analysis of the Zakharov-Shabat (ZS) system associated to the NLS equation, which in turn is represented by an ordinary differential equation of first order \cite{KS}, \cite{KV}. 

In fact, assuming that $u_0 \in L^1(\mathbb{R})$, it can be expressed in the following way:
\begin{equation}\label{ZS}
{\bf i}{\bf J}\frac{\partial {\bf X}}{\partial x}(\zl,x)-{\bf V}(x){\bf X}(\zl,x)=\zl {\bf X}(\zl,x), \quad x \in \RR
\end{equation}
where $\lambda \in \mathbb{C}$ is a spectral parameter, 
\begin{equation*}\label{1.2}
{\bf J}=\begin{pmatrix}1&0\\ 0&-1\end{pmatrix}, \quad {\bf V}(x)=\begin{pmatrix}0&iu_0(x)\\
\pm i \bar{u}_0(x)&0\end{pmatrix}.
\end{equation*}
Here the bar is used to denote complex conjugation.

The initial scattering data are the entries of the so-called scattering matrix and the coefficients and parameters of two spectral sums.
Denoting by
\begin{equation*}\label{S}
{\bf S}(\lambda)=\begin{pmatrix}T(\lambda)&L(\lambda)\\ R(\lambda)& T(\lambda)
\end{pmatrix},
\end{equation*}
the  scattering matrix, $T(\lambda)$ represents  the (initial) transmission coefficient, while $L(\lambda)$ and $R(\lambda)$ stand for the  initial reflection coefficients from the left and from the right, respectively.
If $T(\lambda)$ has no poles in the complex upper half plane $\mathbb{C}^+$, there are no spectral sums to identify.

Otherwise, denoting by $\lambda_1,\, \dots,\, \lambda_n$ the so-called bound states that is the finitely many poles of $T(\lambda)$ in $\mathbb{C}^+$  and by $m_1,\,\dots,\, m_n$ the corresponding multiplicities, we have to identify the coefficients $\{(\Gamma_\ell)_{js}, (\Gamma_r)_{js}\}$ as well the parameters $\{n,m_j,\lambda_j \}$ of the initial spectral sums from the left and from the right
\begin{align}
S_\ell(\alpha)=\sum_{j=1}^{ n} e^{-\lambda_j \alpha} \sum_{s=0}^{{m}_j-1}( \Gamma_\ell)_{js} \frac{\alpha^s}{s!}, \quad \alpha \geq 0 \label{S_ell}\\
S_r(\alpha)=\sum_{j=1}^n e^{\lambda_j \alpha} \sum_{s=0}^{m_j-1}(\Gamma_r)_{js} \frac{\alpha^s}{s!}, \quad \alpha \leq 0 \label{S_r}
\end{align} 
where the coefficients $(\Gamma_\ell)_{js}$ and $(\Gamma_r)_{js}$ are the so-called  norming constants from the left and from the right, respectively, and $0! = 1$.

In the IST technique, a crucial role is played by the initial Marchenko kernels from the left $\Omega_\ell(\alpha)$ and from the right $\Omega_r(\alpha)$, which are connected to the above spectral coefficients and spectral sums as follows: 
\begin{align}
\zO_\ell(\alpha)&=\rho(\alpha)+S_\ell(\alpha), \quad \textrm{for} \quad \alpha \geq 0 \label{Omega_ell} \\ \nonumber \\
\zO_r(\alpha)&=\ell(\alpha)+S_r(\alpha), \quad \textrm{for} \quad \alpha \geq 0 \label{omega_r} 
\end{align}
where 
\begin{align}\label{rho}
\rho(\alpha)= \frac{1}{2\pi}\int_{-\infty}^{+\infty} R(\lambda) e^{i \lambda \alpha} d \lambda =\iftran{R(\lambda)}
\end{align}
is the inverse Fourier transform of the reflection coefficient from the right $R(\lambda)$ 
and 
\begin{align}\label{elle}
\ell(\alpha)= \frac{1}{2\pi}\int_{-\infty}^{+\infty} L(\lambda) e^{-i \lambda \alpha} d \lambda =\frac{1}{2\pi} \ftran{L(\lambda)},
\end{align}
apart from the factor $1/2\pi$, is the Fourier transform of the reflection coefficient from  the left $L(\lambda)$.

We note that $\Omega_\ell(\alpha)$ and $\Omega_r(\alpha)$, respectively, reduce to:
\begin{itemize}
\item[(a)] $S_\ell(\alpha)$ and $S_r(\alpha)$ if the reflection coefficients vanish (reflectionless case);
\item[(b)] $\rho(\alpha)$ and  $\ell(\alpha)$ if there are no bound states.
\end{itemize}

\section{Auxiliary functions}\label{auxiliaryfunctions}

Let us now introduce, for $y \geq x$, the two pairs of unknown auxiliary functions
$$\bar{ {\bf K}}(x,y) \equiv\begin{pmatrix}\bar{K}^{up}(x,y)\\ \bar{K}^{dn}(x,y) \end{pmatrix}, \quad { {\bf K}}(x,y) \equiv\begin{pmatrix}K^{up}(x,y)\\ K^{dn}(x,y) \end{pmatrix},$$ \\
and, for $y \leq x$, the two other pairs of unknown auxiliary functions 
$$\bar{ {\bf M}}(x,y) \equiv\begin{pmatrix}\bar{M}^{up}(x,y)\\ \bar{M}^{dn}(x,y) \end{pmatrix}, \quad { {\bf M}}(x,y) \equiv\begin{pmatrix}M^{up}(x,y)\\ M^{dn}(x,y). \end{pmatrix} $$
Each of these pair of functions, given the initial solution, is the solution of a system of two structured Volterra integral equations \cite{DM2}, \cite{DM1},  \cite{VanDerMee2013}.

More precisely, in the focusing case, which is the case we are addressing in this paper,  for $y \geq x$, the unknown pair $(\bar{K}^{up}, \, \bar{K}^{dn})$ is the solution of the system 

\begin{equation}\label{nucleiKbarrati}
\begin{cases}
\bar{K}^\up(x,y)=-\displaystyle\int_x^\infty u_0(z)\bar{K}^\dn(z,z+y-x) \, dz,\\ 
\bar{K}^\dn(x,y)=\tfrac{1}{2}   \bar{u}_0(\tfrac{1}{2}(x+y))
+\displaystyle\int_x^{\tfrac{1}{2}(x+y)}   \bar{u}_0(z)\bar{K}^\up(z,x+y-z) \, dz, 
\end{cases} 
\end{equation} 
as well as the pair $({K}^{up}, \, {K}^{dn})$ it is of the system 
\begin{equation}
\begin{cases}
K^\up(x,y)=-\tfrac{1}{2} u_0(\tfrac{1}{2}(x+y))
-\displaystyle\int_x^{\tfrac{1}{2}(x+y)}\,u_0(z)K^\dn(z,x+y-z) \, dz, \\ 
K^\dn(x,y)=\displaystyle\int_x^\infty \,  \bar{u}_0(z)K^\up(z,z+y-x) \, dz.
\end{cases}
\end{equation}
Similarly, for $y \leq x$, the unknown pair $(M^{up}, M^{dn})$ is the solution of the system

\begin{equation}\label{nucleiM}
\begin{cases}
M^\up(x,y)=\displaystyle\int_{-\infty}^x u_0(z)M^\dn(z,z+y-x) \, dz \\ 
M^\dn(x,y)=-\tfrac{1}{2}  \bar{u}_0(\tfrac{1}{2}(x+y))
-\displaystyle\int_{\tfrac{1}{2}(x+y)}^x   \bar{u}_0(z)M^\up(z,x+y-z) \, dz \\ 
\end{cases}
\end{equation}
as well as the pair $(\bar{M}^{up}, \bar{M}^{dn})$ it is of the system
\begin{equation}\label{nucleiMbarrati}
\begin{cases}
\bar{M}^\up(x,y)=\tfrac{1}{2}u_0(\tfrac{1}{2}(x+y))
+\displaystyle\int_{\tfrac{1}{2}(x+y)}^x u_0(z)\bar{M}^\dn(z,x+y-z) \, dz \\  
\bar{M}^\dn(x,y)=-\displaystyle\int_{-\infty}^x   \bar{u}_0(z)\bar{M}^\up(z,z+y-x) \, dz.
\end{cases}
\end{equation}

From the computational point of view, it is important to note that  each auxiliary function is uniquely determined on the bisector $y=x$, by the initial solution or its partial integral energy.

In fact, setting $y=x$ in each of the four Volterra systems, we immediately obtain:
\begin{align}
\bar{K}^\dn(x,x)&= \frac 1 2   \bar{u}_0(x)\quad 
\bar{K}^\up(x,x)   =-\frac 12 \int_x^\infty |u_0(z)|^2 \, dz \label{propr1}\\
K^\up(x,x)&= -\frac 1 2 u_0(x), \quad 
K^\dn(x,x)=-\frac{1}{2}\int_x^\infty |u_0(z)|^2 \, dz, \label{propr2} \\
M^\dn(x,x) &=-\frac 1 2   \bar{u}_0(x),\quad  M^\up(x,x) =-\frac{1}{2} \int_{-\infty}^x |u_0(z)|^2 \, dz; \label{propr3} \\
\bar{M}^\up(x,x)&= \frac 1 2 u_0(x),\quad 
\bar{M}^\dn(x,x)= - \frac{1}{2}\int_{-\infty}^x |u_0(z)|^2 \, dz. \label{propr4}
\end{align}

\begin{remark}\label{remark1}
If $u_0(x)$ is real we need to solve uniquely \eqref{nucleiKbarrati} and \eqref{nucleiM} as ${K}^{up}(x,y)=-\bar{K}^{dn}(x,y)$ and ${K}^{dn}(x,y)=\bar{K}^{up}(x,y)$ as well as $\bar{M}^{up}(x,y)={M}^{dn}(x,y)$ and $\bar{M}^{dn}(x,y)=-{M}^{up}(x,y)$.
\end{remark}

\section{Initial Marchenko kernels and scattering matrix}\label{Marchenko}

Once the auxiliary functions have been computed, both $\Omega_\ell$ and $\Omega_r$  can be approximated by solving a Volterra integral equation. In fact, using the strong connection between $\Omega_\ell$ and  the function pairs $(\bar{K}^{up}, \bar{K}^{dn})$ and $(K^{up}, K^{dn})$ as well as that between  $\Omega_r$ and the function pairs 
$(\bar{M}^{up}, \bar{M}^{dn})$ and $(M^{up}, M^{dn})$, for $y \geq x \geq 0$, we have \cite{VanDerMee2013}

\begin{equation}\label{Mar_K}
\bar{K}^{dn}(x,y)+\zO_\ell(x+y)
+\displaystyle \int_x^\infty {K}^{dn}(x,z)\zO_\ell(z+y) \, dz=0, 
\end{equation}
and, for $y \leq x \leq 0$,
\begin{equation}\label{Mar_M}
M^{dn}(x,y)-\zO_r(x+y)
-\displaystyle\int_{-\infty}^x \bar{M}^{dn}(x,z)\zO_r(z+y) \, dz =0.
\end{equation}

Given the auxiliary vectors  $\bar{ K}$,  ${K}$, $\bar{M}$ and $M$,  relations \eqref{Mar_K}-\eqref{Mar_M} can be interpreted as Volterra integral equations having $\Omega_\ell$ and $\Omega_r$  as unknowns. 
\begin{remark}
We point out that, from the computational point of view, each initial Marchenko kernel can be treated as a function of only one variable, as we only have to deal with the sum of the two variables.
\end{remark}


Following the procedure proposed in \cite{VanDerMee2013}, the entries of the scattering matrix $S(\lambda)$ can be computed as follows:
\begin{align}
T(\lambda)& = \frac{1}{a_{\ell 4}(\lambda)}= \frac{1}{a_{r1}(\lambda)},\label{T}\\
L(\lambda)& =\frac{a_{\ell2}(\lambda)}{a_{\ell4}(\lambda)}=-\frac{a_{r2}(\lambda)}{a_{r1}(\lambda)},\label{L}\\
R(\lambda)&=\frac{a_{r3}(\lambda)}{a_{r1}(\lambda)}=-\frac{a_{\ell 3}(\lambda)}{a_{\ell 4}(\lambda)} \label{R}
\end{align}
where
\begin{align*}
\begin{cases}
a_{\ell1}(\lambda)&= 1-\displaystyle\int_{\RR^+} e^{-i \lambda z} \left( \displaystyle\int_{\RR} u_0(y) \overline{K}^\dn(y,y+z) dy \right) dz \\
a_{\ell2}(\lambda)&=-\displaystyle\int_{\RR} e^{2i \lambda y} u_0(y) dy -\displaystyle \int_{\RR} e^{i \lambda z} \left(\displaystyle \int_{-\infty}^{\frac z2 } u_0(y) {K}^{dn}(y,z-y) dy \right) dz \\
a_{\ell3}(\lambda)&=\displaystyle \int_{\RR} e^{-2i \lambda y}   \bar{u}_0(y) dy + \displaystyle \int_{\RR} e^{-i \lambda z} \left(\displaystyle\int_{-\infty}^{\frac z2 }   \bar{u}_0(y) \overline{K}^{up}(y,z-y) dy \right) dz \\
a_{\ell4}(\lambda)&= 1+\displaystyle\int_{\RR^+} e^{i \lambda z} \left(\displaystyle \int_{\RR}   \bar{u}_0(y) K^\up(y,y+z) dy \right) dz
\end{cases} 
\end{align*}
\begin{align*}
\begin{cases}
a_{r1}(\lambda)&= 1+\displaystyle\int_{\RR^+} e^{i \lambda z} \left(\displaystyle \int_{\RR} u_0(y) M^\dn(y,y-z) dy \right) dz \\
a_{r2}(\lambda)&=\displaystyle\int_{\RR} e^{2i \lambda y} u_0(y) dy +\displaystyle \int_{\RR} e^{i \lambda z} \left(\displaystyle\int_{\frac{z}{2}}^{+\infty} u_0(y) \overline{M}^{dn}(y,z-y) dy \right) dz \\
a_{r3}(\lambda)&=-\displaystyle\int_{\RR} e^{-2i \lambda y}   \bar{u}_0(y) dy -\displaystyle \int_{\RR} e^{-i \lambda z} \left(\displaystyle\int_{\frac{z}{2}}^{+\infty}   \bar{u}_0(y) M^{up}(y,z-y) dy \right) dz \\
a_{r4}(\lambda)&= 1-\displaystyle\int_{\RR^+} e^{-i \lambda z} \left( \displaystyle\int_{\RR}   \bar{u}_0(y) \overline{M}^\up(y,y-z) dy \right) dz. 
\end{cases}
\end{align*}

\begin{remark}
If the solution is a soliton or a multisoliton, the only spectral data to be computed are the coefficients and the parameters of the spectral sums $S_\ell(\alpha)$ and $S_r(\alpha)$. The coefficients and parameters of $S_\ell(\alpha)$ can be computed by applying the matrix-pencil method recently proposed in  \cite{FeVanSe2013} to a sufficiently large set of equispaced data of $\Omega_\ell(\alpha)$. 
\end{remark}

%
%
%

\section{Computational strategy}\label{computational}
In this section we propose a numerical procedure to evaluate the initial scattering data in the case 
\begin{equation}\label{cond_u0}
u_0(x)=0, \quad \textit{for} \quad |x|> L.
\end{equation} 
We note that \eqref{cond_u0} can be considered acceptable whenever $u_0(x) \to 0$ for $|x| \to \infty$, provided that $L$ be large enough.

Assuming for computational simplicity that the reflection coefficients are zeros, we must solve:
\begin{itemize}
\item[1.] the systems of Volterra integral equations \eqref{nucleiKbarrati}-\eqref{nucleiMbarrati};
\item[2.] the Volterra integral equations \eqref{Mar_K}-\eqref{Mar_M};
\item[3.] a nonlinear approximation problem \cite{FeVanSe2013}.
\end{itemize} 

\subsection{Auxiliary functions computation} 
Hereafter we assume $u_0$ to be real, though the algorithms remain essentially the same in the complex case. In this case (Remark \ref{remark1}) we need to solve only systems \eqref{nucleiKbarrati} and \eqref{nucleiM}, instead of systems \eqref{nucleiKbarrati}-\eqref{nucleiMbarrati}.

Let us first consider system \eqref{nucleiKbarrati} to identify the supports of $\bar{K}^{up}(x,y)$ and $\bar{K}^{dn}(x,y)$.

\begin{lemma}\label{lemmaK}
Under hyphotesis \eqref{cond_u0}, the following properties hold true:
\begin{itemize}
\item[(1)] For $x>L$, $\bar{K}^{up}(x,y)$ and $\bar{K}^{dn}(x,y)$ are both zero;
\item[(2)] If $x \leq L$ and $x+y \geq 2L$ then $\bar{K}^{up}(x,y)=\bar{K}^{dn}(x,y)=0$;
\item[(3)] If $x<-L$ and $y-x>4L$ then $\bar{K}^{up}(x,y)=0$;
\item[(4)] If $x<-L$ and $x+y<2L$ then $\bar{K}^{dn}(x,y)=0$.
\end{itemize}
\end{lemma}

\begin{proof}
To prove (1) it is sufficient to note that, as $y \geq x$ and then $\frac{1}{2}(x+y) \geq 0$, $u_0(\frac{1}{2}(x+y))=0$ as well as $u_0(z)=0$ for $z \in (x, \frac{1}{2}(x+y))$. About (2), we consider the sequence $x_i=ih$, $i=0,\pm 1, \pm 2,...$ where $h=\frac{L}{n}$, $n \in \mathcal{N}$, and then collocate system \eqref{nucleiKbarrati} in the node $(x_n, x_{n+2})$. Noting that both $u_0(\frac{1}{2}(x+y))=0$ and $u_0(z)=0$ for $z \geq x_{n+1}$, we can write:
\begin{equation}
\begin{cases}
\bar{K}^\up(x_n,x_{n+2})=-\displaystyle\int_{x_n}^{x_{n+1}} u_0(z)\bar{K}^\dn(z,z+2h) \, dz,\\ 
\bar{K}^\dn(x_n,x_{n+2})= \displaystyle\int_{x_n}^{x_{n+1}} u_0(z)\bar{K}^\up(z,x_{2n+2}-z) \, dz.
\end{cases} 
\end{equation} 
Applying then the composite trapezoidal quadrature formula to the computation of the two integrals, and denoting by $\bar{K}^{up}_{n,n+r}$ and $\bar{K}^{dn}_{n,n+r}$, the approximate values of $\bar{K}^{up}(x_n,x_{n+r})$ and $\bar{K}^{dn}(x_n,x_{n+r})$, respectively, we obtain the nonsingular homogeneous system 
\begin{align*}
\begin{cases}
\bar{K}^\up_{n,n+2}+\frac{h}{2} \, u_0(x_n)\bar{K}^\dn_{n,n+2}=0 \\ 
-\frac{h}{2} \, u_0(x_n)\bar{K}^\up_{n,n+2}+\bar{K}^\dn_{n,n+2}=0,
\end{cases} 
\end{align*}
whose solution is $\bar{K}^{up}_{n,n+2}=\bar{K}^{dn}_{n,n+2}=0$, for any fixed $h$ value. Applying recursively the same procedure to the nodal points $(x_j, x_{2n+2\ell-j})$, with $j=n-1,n-2,...$ and $\ell=0,1,2,...$ and considering that $h$ is arbitrary, the result follows immediately. \\
For (3) we note that (2) implies that $\bar{K}^{dn}(z,z+y-x)=0$, as $z \in [-L,L]$ and $y-x>4L$, that is $2z+y-x>2L$.\\ Result (4) is immediate, as the integration domain appearing in the second equation of \eqref{nucleiKbarrati} is null and $u_0(\frac{1}{2}(x+y))=0$.
\end{proof}

Taking into account the above properties, we can say that the supports of the auxiliary functions are those represented in Figure \ref{support1}.
\begin{figure}[t]
\includegraphics[scale=0.36]{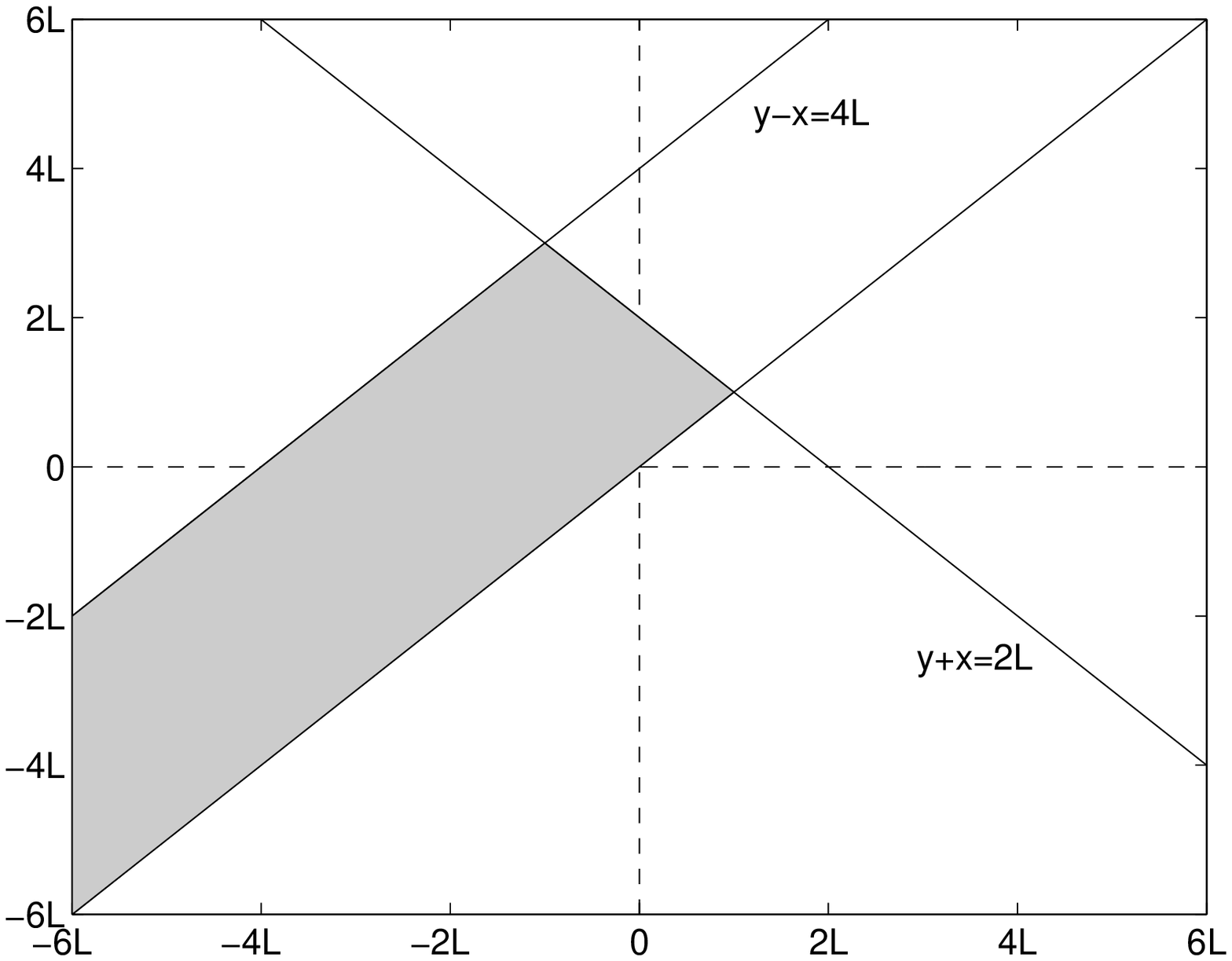}
\includegraphics[scale=0.36]{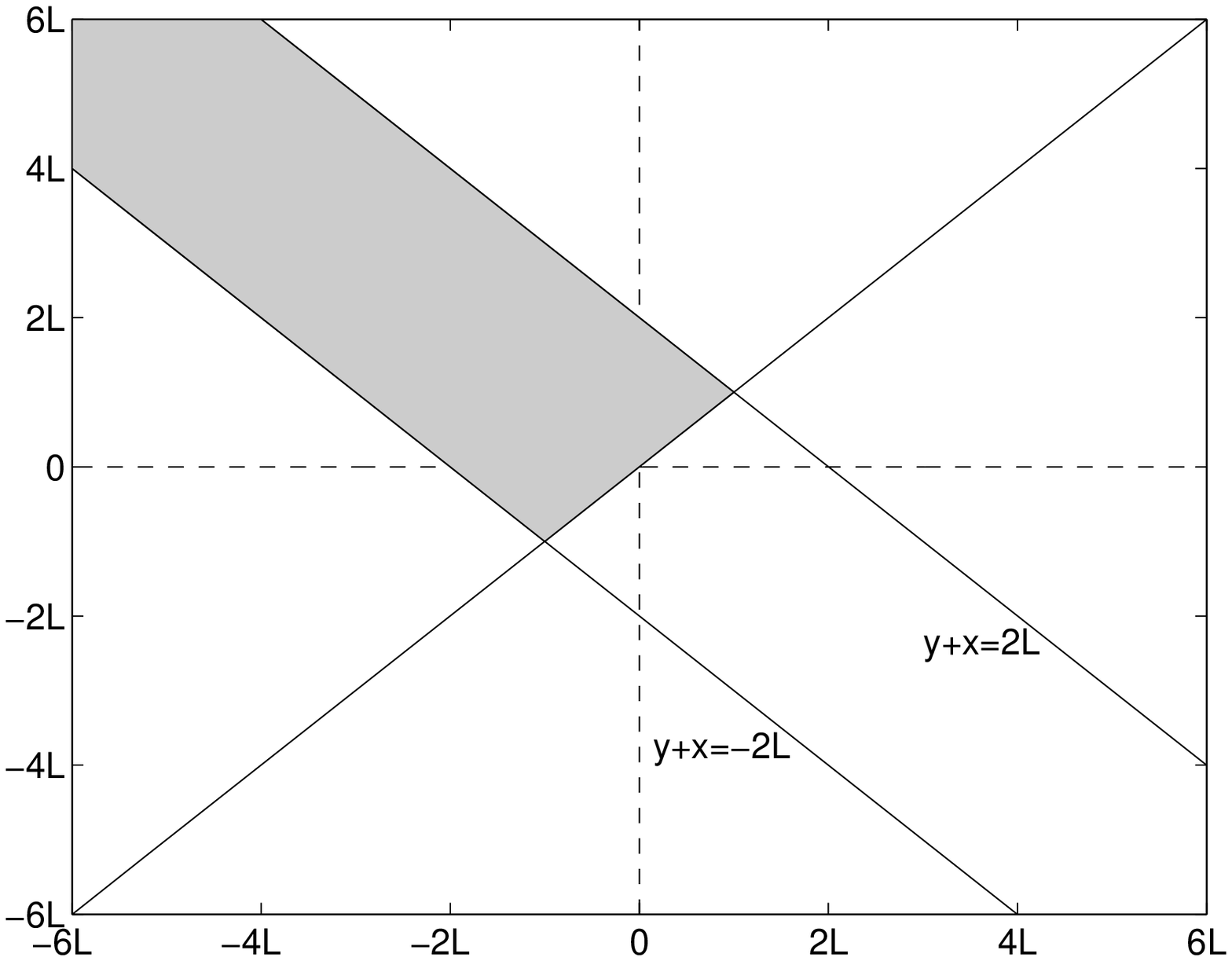}
\caption{\small Supports of the auxiliary functions $\bar{K}^{up}$ \label{support1} (to the left) and $\bar{K}^{dn}$ (to the right)}
\end{figure}

For the numerical solution of system \eqref{nucleiKbarrati}, the following properties are also important:
\begin{itemize}
\item[1.] If $x \leq -L$, whatever $h$, $\bar{K}^{up}(x,y)$ is constant on the line $y=x+h$. This can be seen taking into account that $u_0 \in [-L,L]$ and noting that
$$ \bar{K}^{up}(x,x+h)=-\int_{-L}^L u_0(z) \bar{K}^{dn}(z,z+h) dz $$
only depends on $h$. For this reason we put $\bar{K}^{up}(x,x+h)= {\mathcal{C}}^{up}_{\bar{K},h}$ for each given value $h$.
\item[2.] If $x<-L$ and $x+y>-2L$, $\bar{K}^{dn}(x,y)$ is constant on each line $x+y=-2(L+h)$ for each $0<h<2L$. In fact, by the second equation of \eqref{nucleiKbarrati} it results that 
$$ \bar{K}^{dn}(x,y)=\frac{1}{2}u_0(-L+h)+\int_{-L}^{L+h} u_0(z) \bar{K}^{up}(z,-2(L+h)-z) dz = {\mathcal{C}}^{dn}_{\bar{K},h}.$$
\end{itemize}
\begin{figure}[t]
\includegraphics[scale=0.36]{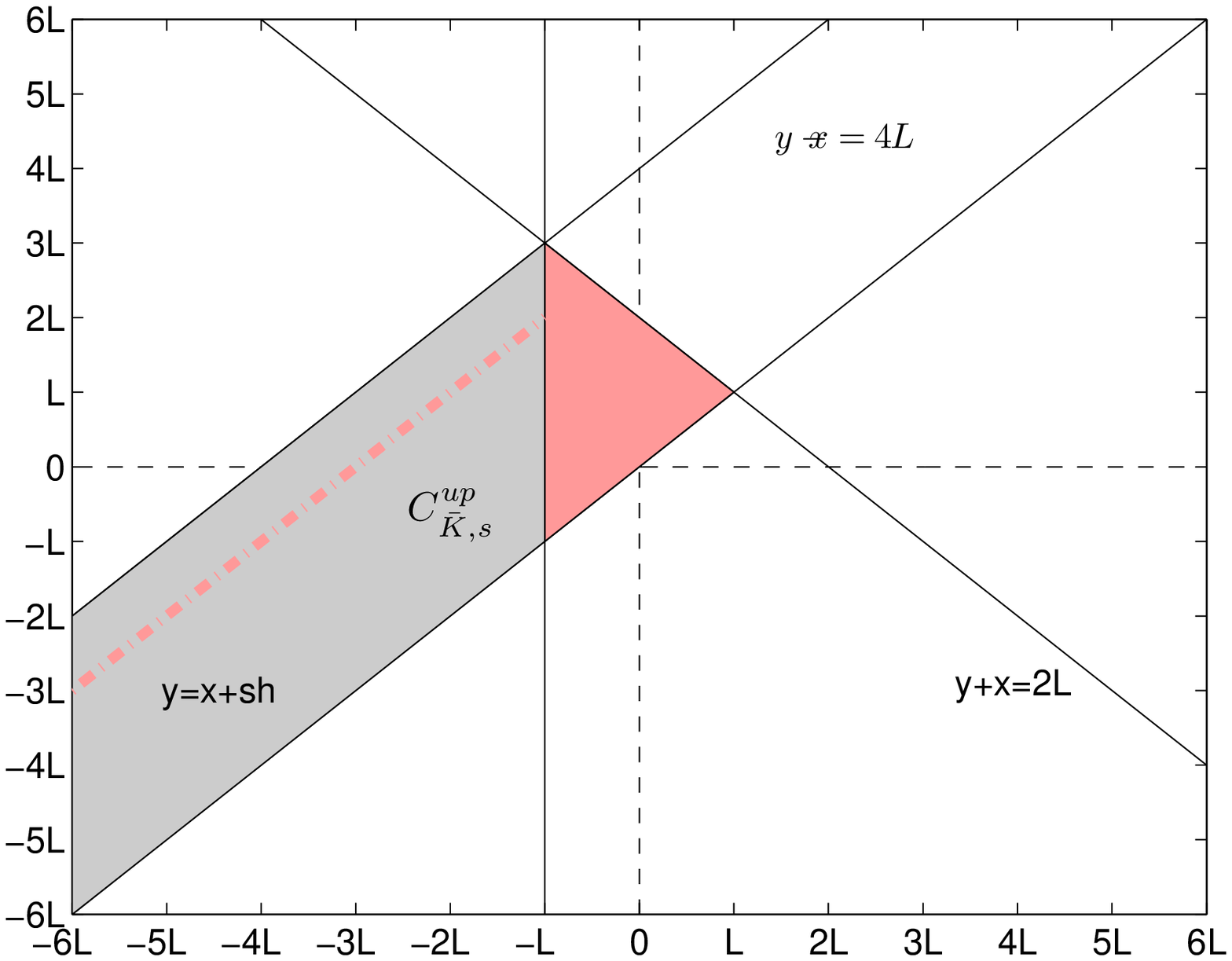}
\includegraphics[scale=0.36]{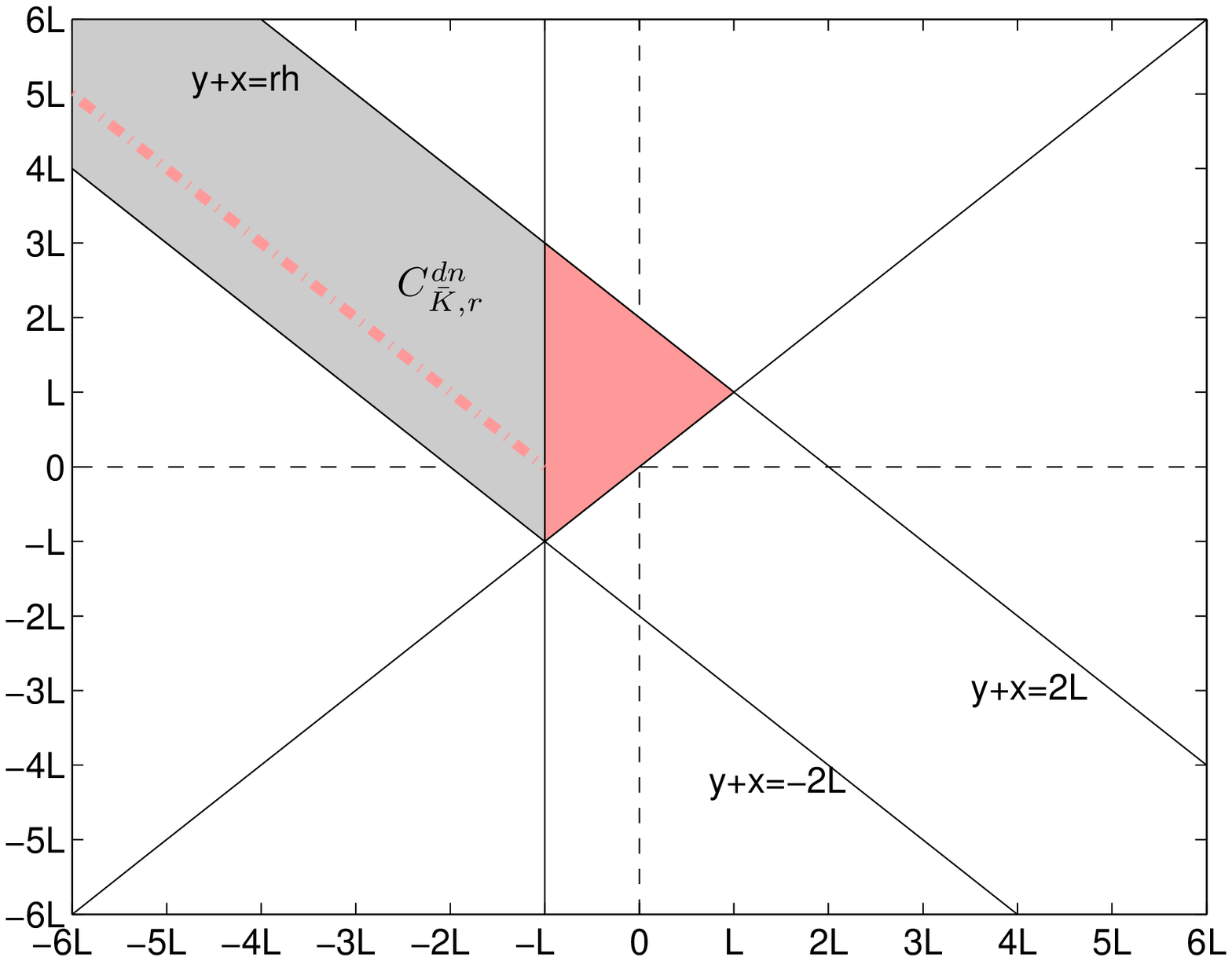}
\caption{Additional properties of $\bar{K}^{up}$ (to the left) and $\bar{K}^{dn}$ (to the right)   \label{property1}}
\end{figure}
These two results are grafically represented in Figure \ref{property1}.

A visualization of the area where we need to compute $\bar{K}^{up}$ and $\bar{K}^{dn}$ is given by the orange triangle represented in  Figure \ref{square}. In the remaining areas of the respective supports their values are immediately obtained by using those of the orange triangle. The orange line shows, in particular, the values of the orange triangle we use to know $\bar{K}^{up}$ and $\bar{K}^{dn}$ in the orange point of the gray area.
\begin{figure}[t]
\includegraphics[scale=0.5]{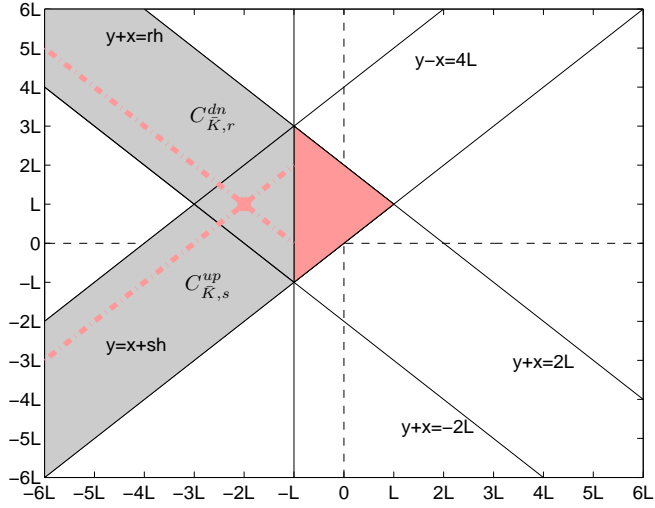}
\caption{Geometrical visualization of the computational area \label{square}}
\end{figure}

Fixed $n \in \mathbb{N}$, let $x_i=ih$ $i=-n,-n+1,...,n$ and $y_j=jh$ with $h=\frac{L}{n}$, \, $j=-n,-n+1,...,3n$. The algorithm first requires the computation of $\bar{K}^{up}$ and $\bar{K}^{dn}$ in the nodal points of the bisector $(x_i,y_j)$ of the orange triangle.
Recalling \eqref{propr1} and denoting by $\bar{K}^{up}_{r,s}$ and $\bar{K}^{dn}_{r,s}$ the approximation of $\bar{K}^{up}(x,y)$ and $\bar{K}^{dn}(x,y)$ in the nodal points $(x_r,y_s)$, we can write:
\begin{align*}
\bar{K}^\dn_{n-i+1,n-i+1}&=\frac 12 u_{0,n-i+1}, \\
\bar{K}^\up_{n-i+1,n-i+1} &= -\frac 12 \int_{{x}_{n-i+1}}^L |u_0(z)|^2 \, dz, i=1,2,...,2n+1.
\end{align*}

The approximation of $\bar{K}^{up}_{n-i+1,n-i+1}$ is then obtained recursively by applying the composite Simpson's quadrature formula.
Fixed $j=1,2,\dots,2n$, system  \eqref{nucleiKbarrati} is collocated in the nodal points 
$$D_j=\{ (x_{n-j-i+1}, x_{n+j-i+1}) \, \textrm{with}  \quad \forall i=1,\ ...\ , \ 2n+1 \}.$$

We note that in each collocation point we have to compute $\bar{K}^{up}$ and $\bar{K}^{dn}$, given their values in the bisector $y=x$. For this reason, recalling that $\bar{K}^{up}(x_i,y_i)$ and $\bar{K}^{dn}(x_i,y_i)$ are zero for $x_i+y_j>2L$, we collocate system  \eqref{nucleiKbarrati} following the order depicted in Figure \ref{collocazione}.

Our numerical algorithm is based on the approximation of the integrals \eqref{nucleiKbarrati} by means of the composite trapezoidal formula. In this way, following the ordering depicted in Figure \ref{collocazione}, for each collocation point we have only to solve recursively a sequence of nonsingular $2 \times 2$ linear systems. In fact, for each collocation point, we must solve the following system:
\begin{align*}
\begin{cases}
\bar{K}^{up}_{n-j-i+1,n+j-i+1} + \frac{h}{2} u_{0,n-j-i+1} \bar{K}^{dn}_{n-j-i+1,n+j-i+1} \\ \hspace{3cm}= -h(1-\delta_{i-1,0}) \displaystyle \sum_{\ell=1}^{i-1} u_{0,n-j-i+1+\ell} \bar{K}^{dn}_{n-j-i+\ell+1,n+j-i+\ell+1} \\ 
- \frac{h}{2} u_{0,n-j-i+1} \bar{K}^{up}_{n-j-i+1,n+j-i+1} + \bar{K}^{dn}_{n-j-i+1,n+j-i+1} \\ \hspace{3cm} = \dfrac{u_{0,n-i+1}}{2}+\frac{h}{2} \displaystyle \sum_{\ell=1}^{j} (2-\delta_{\ell j}) u_{0,n-j-i+\ell+1} \bar{K}^{up}_{n-j-i+\ell+1,n+j-i-\ell+1}.
\end{cases}
\end{align*}

\begin{figure}[t!]
\includegraphics[scale=0.33]{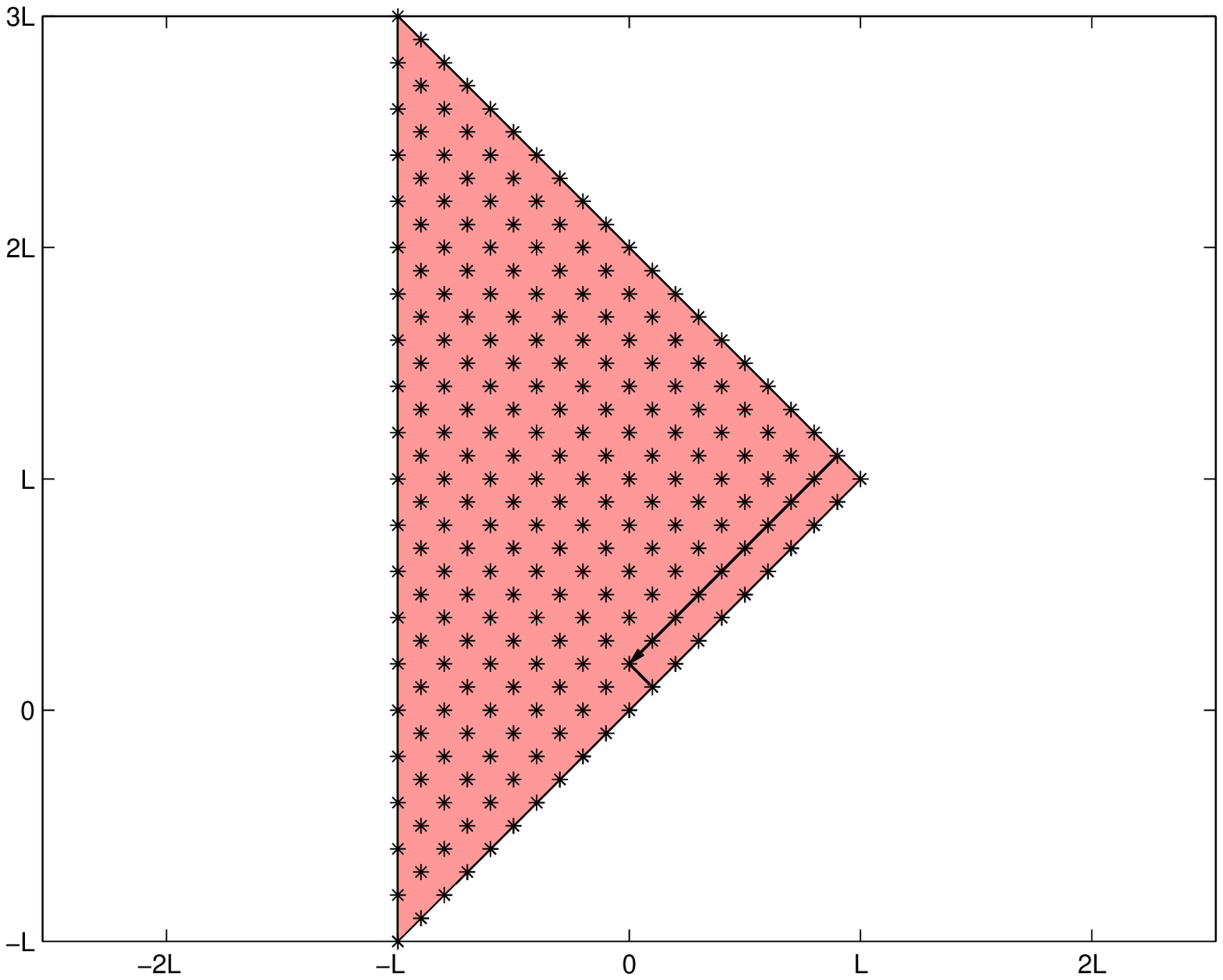} 
\includegraphics[scale=0.33]{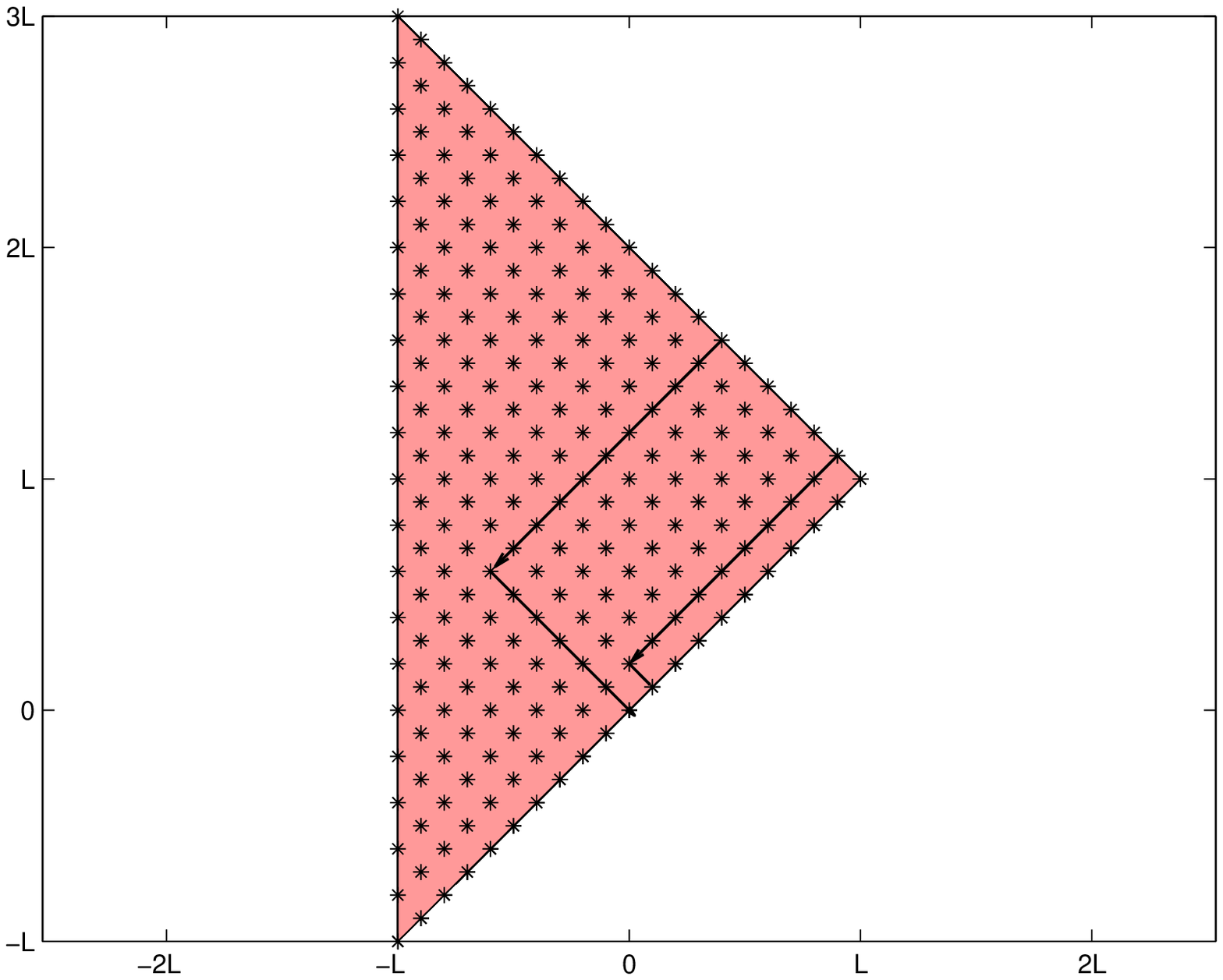} 
\caption{Sorting visualization of collocation points  in the orange trangle \label{collocazione}}
\end{figure}

Let us now consider system \eqref{nucleiM} and identify the supports of the auxiliary functions $M^{up}$ and $M^{dn}$.
Taking into account hypothesis \eqref{cond_u0}, they are characterized by the following 

\begin{lemma}\label{lemmaM}
Under hypothesis \eqref{cond_u0}, the following statements hold true:
\begin{itemize}
\item[(1)] For $x<-L$ both $M^{up}(x,y)$ and $M^{dn}(x,y)$ are zero;
\item[(2)] For $x \geq -L$ and $x+y \leq -2L$, $M^{up}(x,y)=M^{dn}(x,y)=0$;
\item[(3)] If $x>L$ and $y-x<-4L$ then $M^{up}(x,y)=0$;
\item[(4)] If $x>L$ and $y+x>2L$, then $M^{dn}(x,y)=0$.
\end{itemize} 
\end{lemma}
\begin{proof}
About (1) it is enough to note that $\frac{1}{2}(x+y)<-L$ and then $M^{up}(x,y)=M^{dn}(x,y)=0$. \\ In order to prove (2), we observe that collocating system \eqref{nucleiM} in $(x_{-n}, x_{-n-2})$ and taking into account that $u_0(\frac{1}{2}(x+y))=0$ and $u_0(z)=0$ for $z<x_{-n}$, one can write
\begin{align*}
\begin{cases}
{M}^\up_{-n,-n-2}-\frac{h}{2} \, u_0(x_n){M}^\dn_{-n,-n-2}=0 \\ 
\frac{h}{2} \, u_0(x_{-n}){M}^\up_{-n,-n-2}+{M}^\dn_{-n,-n-2}=0,
\end{cases} 
\end{align*}
whose solution is $M^{up}_{-n,-n-2}=M^{dn}_{-n,-n-2}$ for any given $h$ value. Iterating the procedure in the nodal points $(x_{-j},x_{-j-2\ell})$ with $j=n-1,n-2,...$ and $\ell=1,2,...$ and considering that $h$ is arbitrarily, we obtain that $M^{up}_{-j,-j-2\ell}=M^{dn}_{-j,-j-2\ell}=0$, for the above fixed value of $j$ and $\ell$.\\ About the assertion (3), it is sufficient to note that, as the domain of integration is $[-L,L]$ and $2z+y-x<-2L$, $M^{up}(x,y)=0$ as a consequence of (2). \\
Finally (4) is immediate considering that $M^{dn}(x,y)=0$ for $x+y>2L$ as well as $u_0(\frac{1}{2}(x+y))$.
\end{proof}

A geometrical representation of these supports is given in Figure \ref{support2}. 
\begin{figure}[t!]
\includegraphics[scale=0.36]{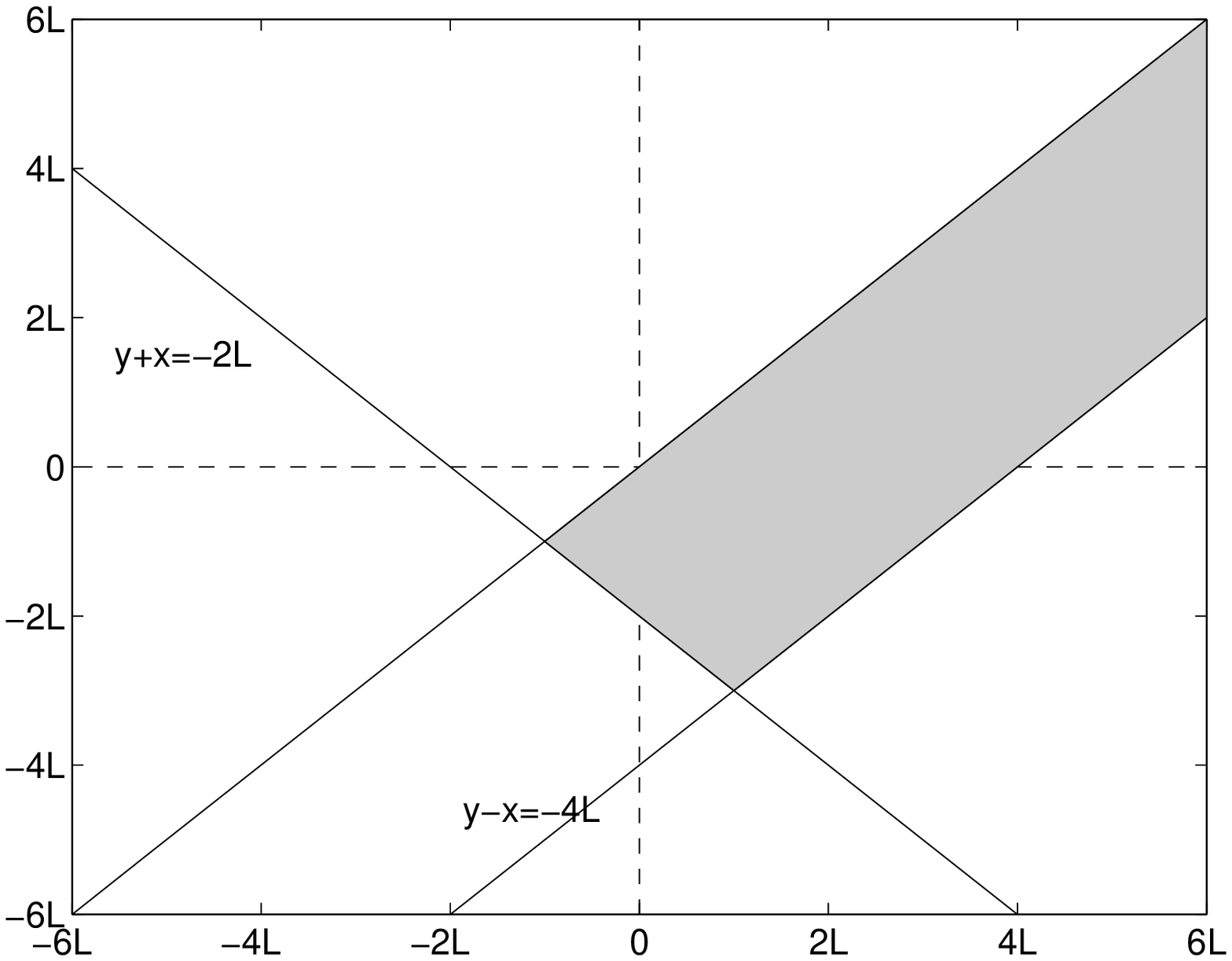}
\includegraphics[scale=0.36]{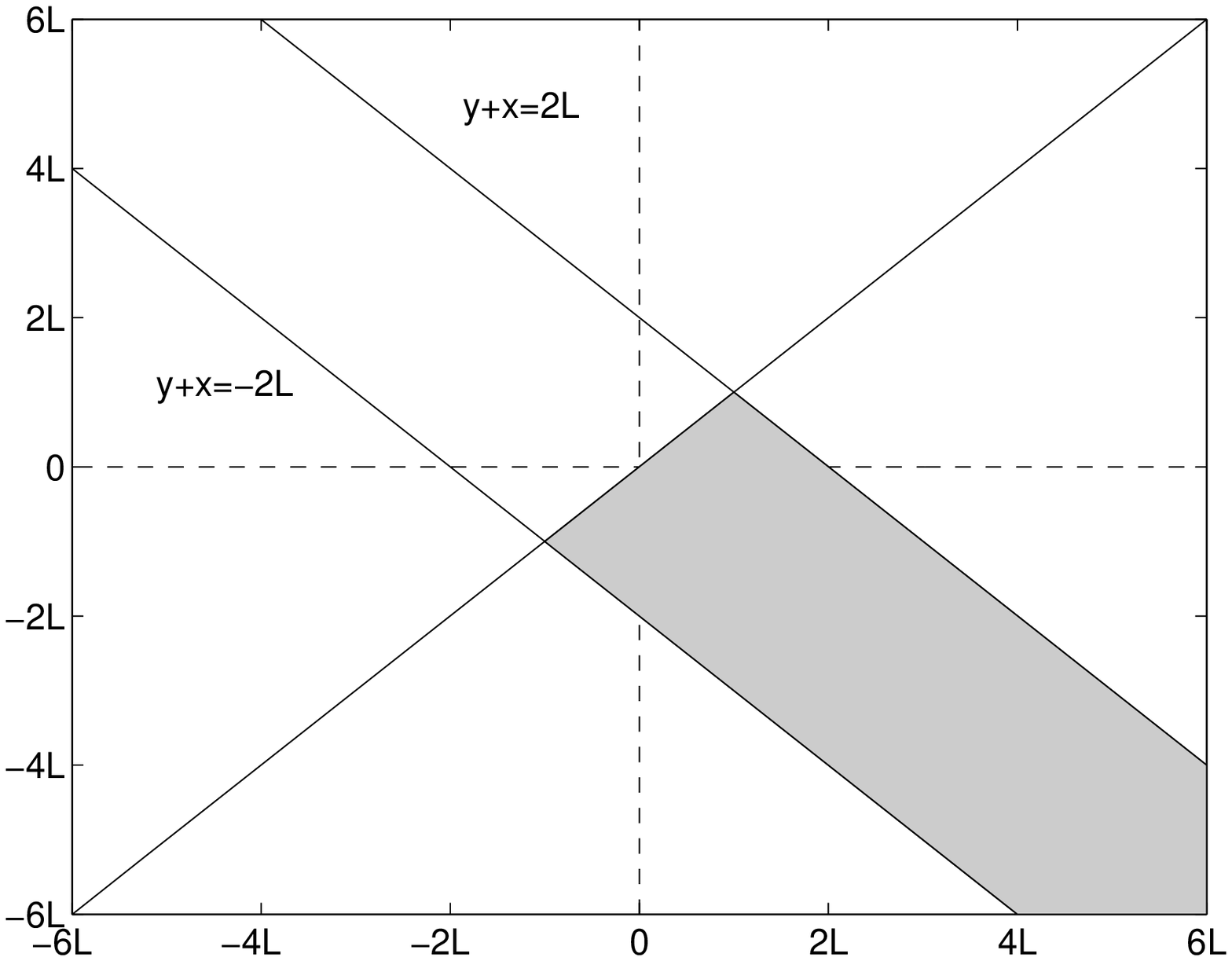}
\caption{\small Supports of the auxiliary functions ${M}^{up}$ (to the left) and ${M}^{dn}$ (to the right) \label{support2}}
\end{figure}

Moreover, in order to find the numerical solution of \eqref{nucleiM}, it is important to note that the unknown functions $M^{up}$ and $M^{dn}$ have the following further properties, represented in Figure \ref{property2}:
\begin{itemize}
\item[1.] If $x \geq L$, whatever $h$, $M^{up}$ is constant on the line $y=x+h$. This property can immediately be verified, noting that
$$ M^{up}(x,y)=\int_{-L}^L u_0(t) M^{dn}(z,z+h) dz$$
only depends on $h$. Hence, to make evident this property, for $x \geq L$, we write $M^{up}(x,y)={\mathcal{C}}^{up}_{M,h}$. 
\item[2.] If $x \geq L$ and $x+y<2L$, $M^{dn}$ is constant on each line $x+y=2(L-h)$, for $0<h<2L$. In fact using the second equation of \eqref{nucleiM}, we have
$$M^{dn}(x,y)=-\frac{1}{2} u_0(L-h)-\int_{L-h}^{L} u_0(z) M^{up}(z,2(L-h)-z) dz= {\mathcal{C}}^{dn}_{M,h}.$$
\end{itemize}
\begin{figure}[t!]
\includegraphics[scale=0.36]{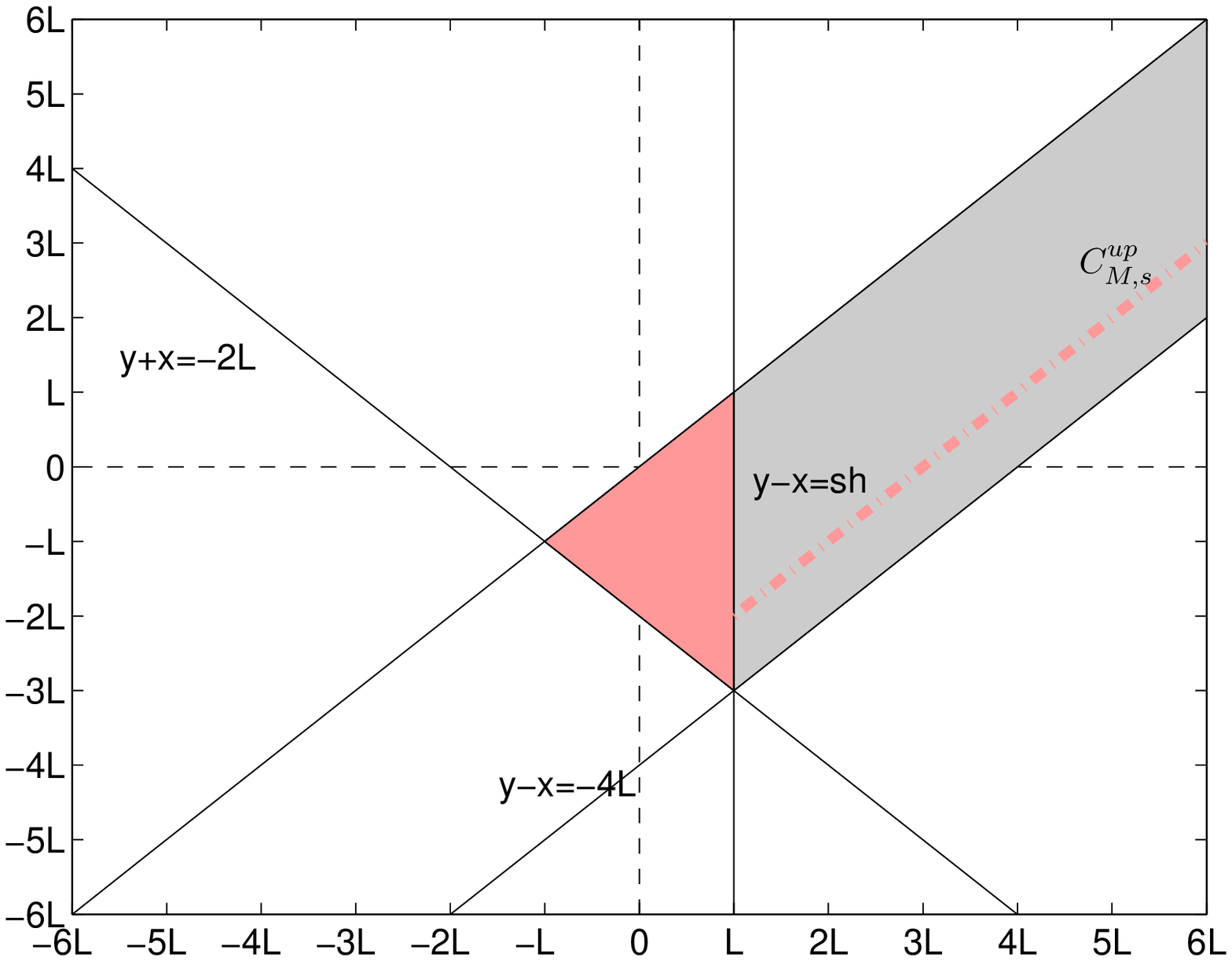}
\includegraphics[scale=0.36]{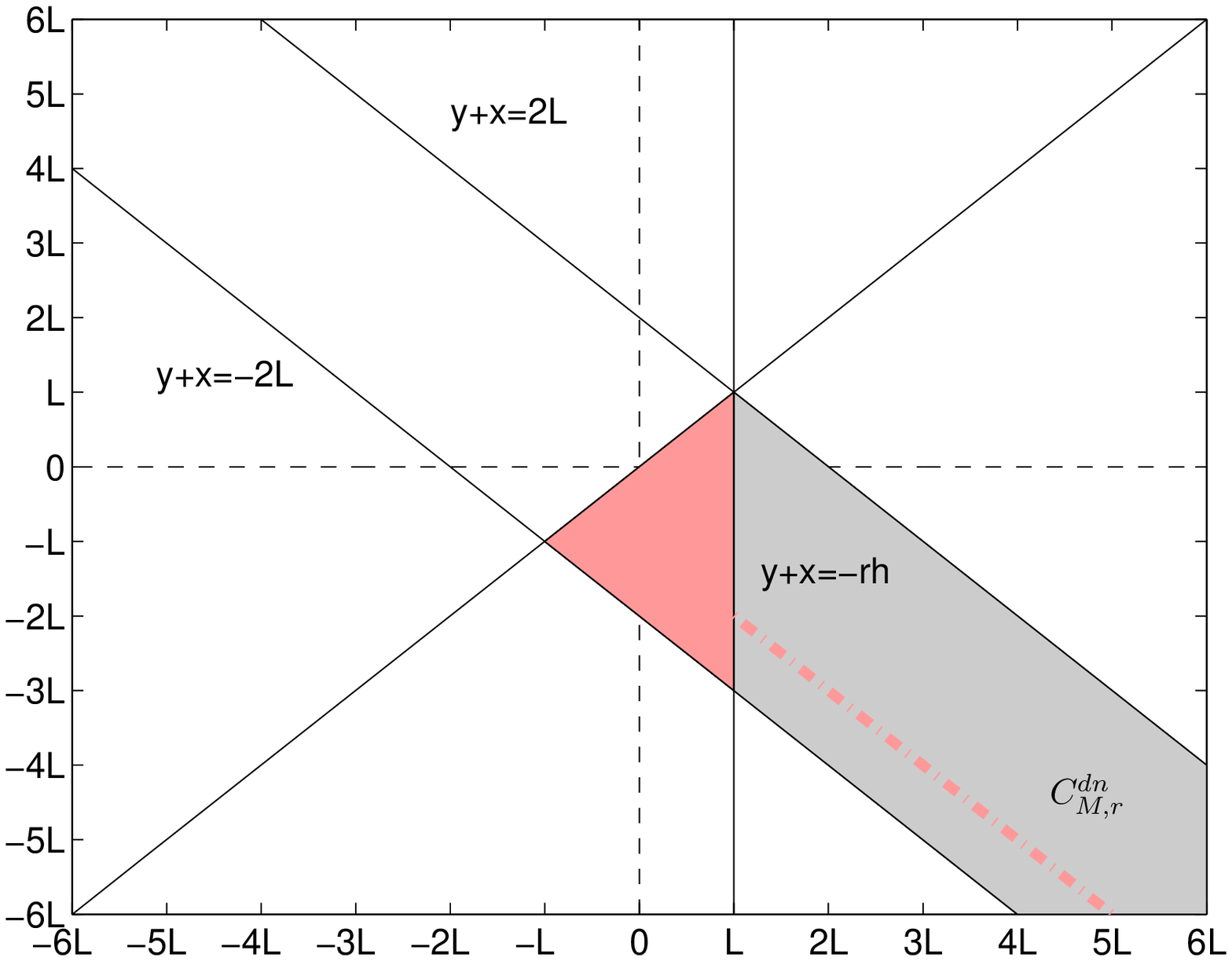}
\caption{Additional properties of ${M}^{up}$ (to the left) and ${M}^{dn}$ (to the right)   \label{property2}}
\end{figure}

As a result, as for system \eqref{nucleiKbarrati}, we need only to compute the unknowns in the orange triangle depicted in Figure \ref{square2}, since in the remainder of the support we can apply the properties discussed above.
\begin{figure}[t!]
\includegraphics[scale=0.5]{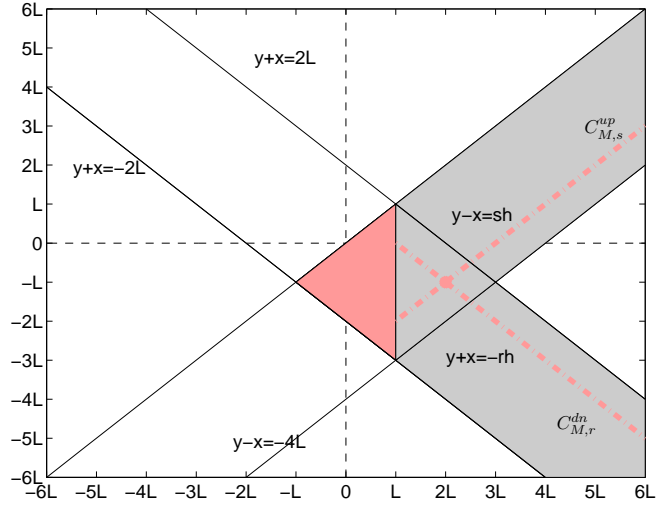}
\caption{Geometrical visualization of the computational area \label{square2}}
\end{figure}

The computational strategy developed for the numerical solution of this system is essentially the same adopted for system \eqref{nucleiKbarrati}. Hence, at first by using \eqref{propr3} we compute  
\begin{align*}
M^\up_{n-i+1,n-i+1}& = -\frac 12 \int_{-\infty}^{x_{n-i+1}} |u_0(z)|^2 dz \\
M^\dn_{n-i+1,n-i+1} &=- \frac 12 u_{0,n-i+1}. 
\end{align*}
 
After that, fixed $j=1,2,\dots,2L$, system \eqref{nucleiM} is collocated in the nodal points 
$$D_j=\{ (x_{-n+j+i-1}, x_{-n-j+i-1}) \in \RR^2 \  \quad \forall i=1,\ ...\ , \ 2n-j+1,  \forall j=1,2,3,... \},$$
and the integrals are approximated by using the composite trapezoidal rule. Operating in this way, we obtain the sequence of $2 \times 2$ systems
\begin{equation*}
\begin{cases}
{M}^{up}_{i+j-n-1,i-j-n-1} -\frac{h}{2} u_{0,i+j-n-1} {M}^{dn}_{i+j-n-1,i-j-n-1}  \\ \hspace{3cm}   =  h (1-\delta_{i-1,0}) \displaystyle \sum_{\ell=1}^{i-1} u_{0,i+j-n-1-\ell} {M}^{dn}_{i+j-n-\ell-1,i-j-n-\ell-1} \\ \\
\frac{h}{2} u_{0,i+j-n-1}  {M}^{up}_{i+j-n-1,i-j-n-1} + {M}^{dn}_{i+j-n-1,i-j-n-1} \\ \hspace{3cm} = -\frac{u_{0,i-n-1}}{2}-\frac{h}{2} \displaystyle \sum_{\ell=1}^{j} (2-\delta_{\ell j}) u_{0,i+j-n-\ell-1} {M}^{up}_{i+j-n-\ell-1,i+\ell-j-n-1}
\end{cases}
\end{equation*}
that we solve recursively, by following the ordering depicted  in Figure \eqref{collocazioneM}.

\begin{figure}[t!]
\includegraphics[scale=0.33]{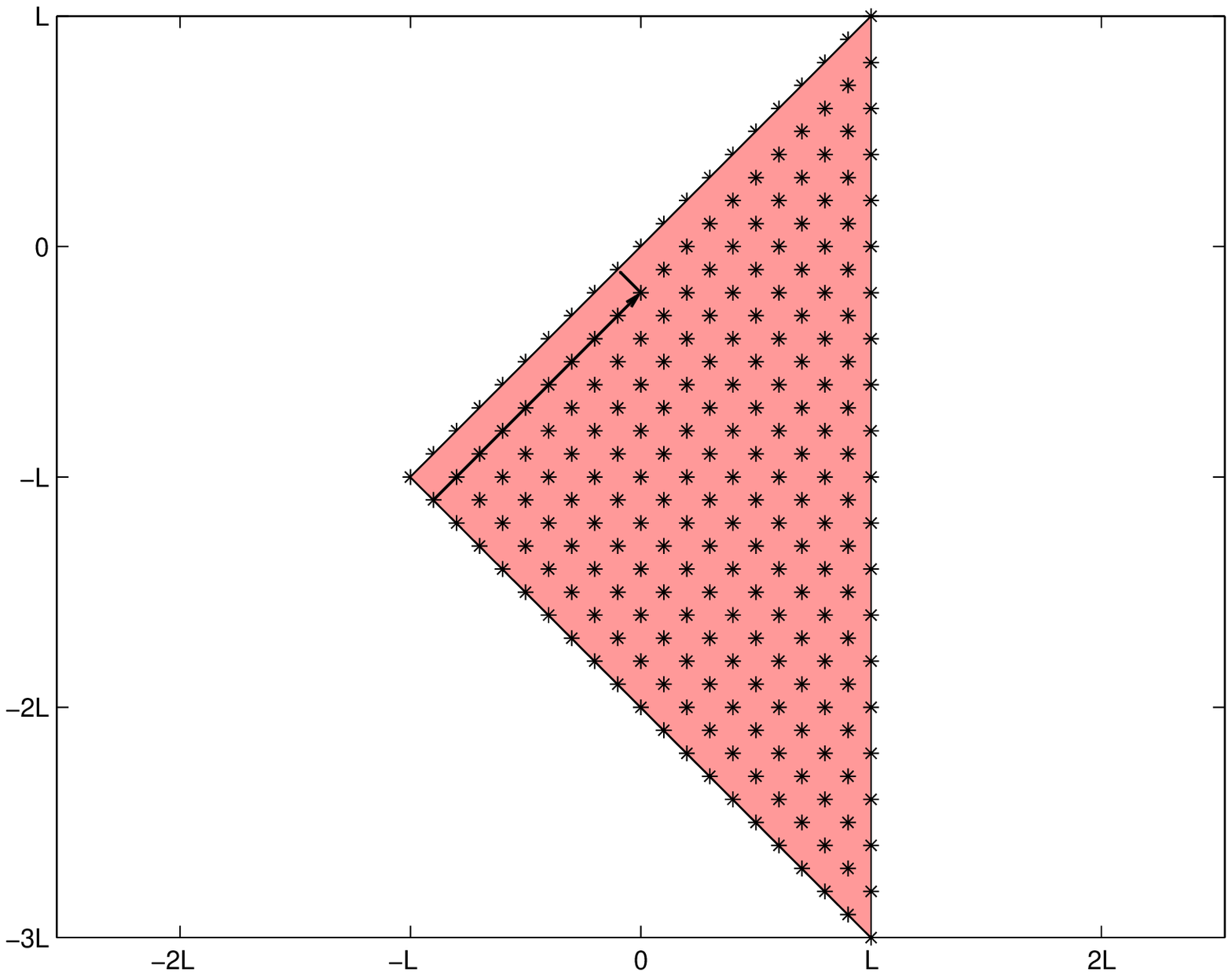}
\includegraphics[scale=0.33]{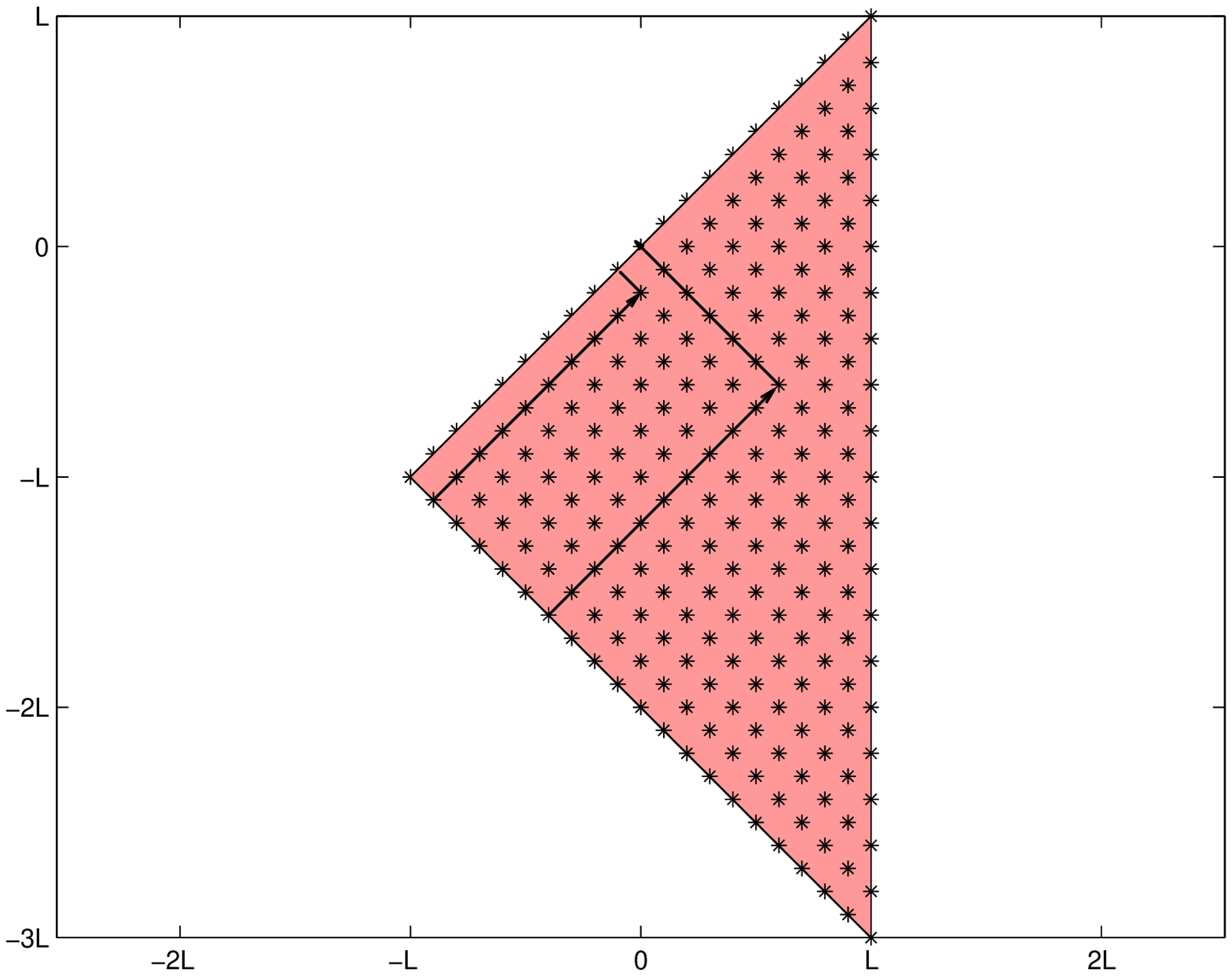} 
\caption{Sorting visualization of collocation points in the orange trangle \label{collocazioneM}}
\end{figure}

\subsection{Numerical solution of Marchenko equations}
Once $\bar{K}^{up}$ and $\bar{K}^{dn}$ has been computed, we have to solve the integral equations 
\eqref{Mar_K} and \eqref{Mar_M} that, by Remark \ref{remark1},  can be written as follows:

\begin{equation}\label{Mar_K1}
\bar{K}^{dn}(x,y)+\zO_\ell(x+y)
+\displaystyle \int_x^\infty \bar{K}^{up}(x,z)\zO_\ell(z+y) \, dz=0
\end{equation}
and
\begin{equation}\label{Mar_M1}
M^{dn}(x,y)-\zO_r(x+y)
+\displaystyle\int_{-\infty}^x M^{up}(x,z)\zO_r(z+y) \, dz =0,
\end{equation}
for $y \geq x \geq 0$ and $y \leq x \leq 0$, respectively.
Let us first consider \eqref{Mar_K1}. Fixed a steplenght $h$ and taken a set of nodal points $(x_i,x_j)$ with $x_i=ih$, $i=0,1,...$ and $j=i,i+1,...$ we collocate \eqref{Mar_K1} in $(x_i,x_j)$, taking a great advantage by the support of $\bar{K}^{up}$ and  $\bar{K}^{dn}$ as explained in the following

\begin{lemma}
If $u_0(x)=0$ for $|x|>L$, $\Omega_\ell(z)=0$ for $z>2L$.
\end{lemma}
\begin{proof}
Given the  steplenght $h$ and assuming $x_n=nh=L$, we set  $\bar{K}^{dn}(x_i,x_j)=\bar{K}^{dn}_{ij}$ and $\bar{K}^{up}(x_i,x_j)=\bar{K}^{up}_{ij}$ as well as $\Omega_\ell(x_i+y_j)=(\Omega_\ell)_{i+j}$. Recalling that, as proved in Lemma \ref{lemmaK},
$$\bar{K}^{dn}_{n,n+m}=\bar{K}^{up}_{n,n+m}=0, \quad m=1,2,...$$
and using \eqref{Mar_K1} it is immediate to state that
$$(\Omega_\ell)_{2n+m}=0, \quad m=1,2,...$$
Moreover,  considering that $\bar{K}^{up}(x,z)=0$ for $x \geq x_n$  and recalling \eqref{propr1}, equation \eqref{Mar_K1} implies that 
\begin{equation}\label{Omega_2n}
(\Omega_\ell)_{2n}=(\Omega_\ell)(x_n,x_n)=-\bar{K}^{dn}_{n,n}=\frac{1}{2} u_0(x_n).
\end{equation}
As a result, if the support of $u_0$ is $[-L,L]$ then the support of $\Omega_\ell(\alpha)$ is $[0,2L]$.

Analogous considerations allow to claim that
$$ supp(\Omega_r(z)) \subset [-2L,0]$$
if $supp(u_0) \subset [-L,L]$.
\end{proof}

To compute $\Omega_\ell$ in $[0,2L]$, first we collocate \eqref{Mar_K1} in $(x_{n-1},x_n)$ and approximate the integral by the trapezoidal rule. Proceeding in this way we obtain that $(\Omega_\ell)_{2n-1}$ is the solution of the equation
$$\left(1-\frac{h}{2} \bar{K}^{up}_{n-1,n-1}\right) (\Omega_\ell)_{2n-1}=-\bar{K}^{dn}_{n-1,n}+\frac{h}{4} \bar{K}^{up}_{n-1,n} u_0(x_n).$$
Collocating now in $(x_{n-1}, x_{n-1})$ and adopting the same procedure we obtain that for $h$ small enough $(\Omega_\ell)_{2n-2}$, is the solution of the equation
$$\left(1-\frac{h}{2} \bar{K}^{up}_{n-1,n-1}\right) (\Omega_\ell)_{2n-2}=-\bar{K}^{dn}_{n-1,n-1}+\frac{h}{2} \bar{K}^{up}_{n-1,n} (\Omega_\ell)_{2n-1}.$$
Iterating the procedure and collocating \eqref{Mar_K1} in $(x_{n-j},x_{n-j+k})$ with $j=1,2,...,n$ and  $k=0,1$ (if $j=n$ then $k=0$) we obtain that $(\Omega_\ell)_{2(n-j)+k}$
is the solution of the following  equation
$$\left(1-\frac{h}{2} \bar{K}^{up}_{n-j,n-j}\right) (\Omega_\ell)_{2(n-j)+k}=-\bar{K}^{dn}_{n-j,n-j+k}+\frac{h}{2} \sum_{\ell=1}^{j-1} \bar{K}^{up}_{n-j,n-j+\ell} (\Omega_\ell)_{2(n-j)+k+\ell}.$$
As $|\bar{K}^{dn}(x,x)|$ is decreasing, $h$ can be considered sufficiently small if $ 1-\frac{h}{2} \bar{K}^{up}_{0,0}>0$.

The same approach allows us to compute $\Omega_r$, that is to solve \eqref{Mar_M1} in a sequence of discretization points of its support $x_i=ih$, $i=0,-1,...,-2n$, $x_{-n}=-L$. To this end, recalling that $M^{up}(x,z)=0$ if $x+z<-2L$ and using \eqref{Mar_M1}, first we note that in the bisector $y=x$
$$(\Omega_r)_{-2n}=\Omega_r(x_{-n},x_{-n})=M^{dn}_{-n,-n}=-\frac{1}{2}u_0(x_{-n}). $$
Collocating \eqref{Mar_M1} in $(x_{-n+1},x_{-n})$ and approximating the integral value by the trapezoidal rule we obtain the equation
$$\left(1+\frac{h}{2} {M}^{up}_{-n+1,-n+1}\right) (\Omega_r)_{-2n+1}={M}^{dn}_{-n+1,-n}+\frac{h}{4} {M}^{up}_{-n+1,-n} u_0(x_{-n}).$$
Iterating the procedure we collocate \eqref{Mar_M1} in $(x_{-n+j},x_{-n+j-k})$ with $j=1,2,...,n$ and  $k=0,1$ (if $j=n$ then $k=0$). Hence we obtain that $(\Omega_r)_{2(j-n)-k}$ is solution of the following equation
\begin{align*}
\left(1+\frac{h}{2} {M}^{up}_{-n+j,-n+j}\right) (\Omega_r)_{2(j-n)-k}& \\ &\hspace{-2cm} =  M^{dn}_{-n+j,-n+j-k}+\frac{h}{2} \sum_{\ell=1}^{j-1} M^{up}_{j-n,j-n-\ell} (\Omega_r)_{2(j-n)-k-\ell},
\end{align*}
for any $h$ small enough, that is for $1+\frac{h}{2} M^{up}_{0,0}>0$.

\subsection{Bound states and norming constants}
Let us give a brief description of the matrix-pencil method that we have recently developed for the identification of the bound states and the norming constants \cite{FeVanSe2013}.
Setting $z_j=e^{f_j}$, the spectral function sum $S_\ell(\alpha)$ introduced in \eqref{S_ell}, can be represented as the monomial-power sum
\begin{equation*}
S_\ell(\alpha)=\sum_{j=1}^{n} \sum_{s=0}^{m_j-1} c_{js} \alpha^s z_j^\alpha, \quad 0^0 \equiv 1.
\end{equation*}

Setting $M=m_1+...+m_n$, the method allows to compute the parameters $\{n,m_j,z_j\}$ and the coefficients $\{c_{js}\}$  given $S_\ell(\alpha)$ in $2N$ integer values ($N>M$) 
$$\alpha=\alpha_0,\alpha_0+1,\dots, \alpha_0+2N-1, \quad \text{with} \quad \alpha_0 \in \mathbb{N}^+=\{0,1,2,... \}.$$ 

The basic idea of the method is the interpretation of $S_\ell(\alpha)$  as the general solution of a homogeneous linear difference equation of order $M$
$$ \sum_{k=0}^M p_k S_{k+\alpha_0}=0$$
whose characteristic polynomial (Prony polynomial)
$$P(z)=\prod_{j=1}^{n} (z-z_j)^{m_j} =\sum_{k=0}^M p_k z^k, \quad p_M \equiv 1$$
is uniquely characterized by the $z_j$ values we are looking for. The identification of the zeros $\{ z_j\}$ allows to compute the  coefficients $c_{js}$ by solving a linear system.

For the computation of $\{z_j\}$ and then of the bound states $\lambda_j$, the given data are arranged in the two Hankel matrices of order $N$
$$({\bf{S}}_{\ell}^{0})_{ij}= S_\ell(i+j-2), \quad ({\bf{S}}_{\ell}^{1})_{ij}=S_\ell(i+j-1), \quad \quad i,j=1,2,\dots,N. $$

To these matrices we associate the $M \times M$ matrix-pencil
$$ {\bf{S}}_{MM}(z)=({{\bf{S}}_{NM}^{0}})^* ({\bf{S}}_{NM}^{1} - z {\bf{S}}_{NM}^{0})$$
where the asterisk denotes the conjugate transpose.
As proved in \cite{FeVanSe2013}, the zeros $z_j$ of the Prony polynomial, with their multiplicities, are exactly the generalized eigenvalues of the matrix-pencil ${\bf{S}}_{MM}(z)$. Then, by applying the Generalized Singular Value Decomposition to the matrices  ${\bf{S}}_{NM}^{0}$ and ${\bf{S}}_{NM}^{1}$, the algorithm developed allows to compute the zeros $z_j$ and then the bound states $\lambda_j$, as $\lambda_j=-\log{z_j}.$

The vector of coefficients  $${\bf{c}}=[c_{1\,0},...,c_{1\,n_1-1},...,c_{L\,0},...,c_{L\,n_{M}-1}]^T$$ is then  computed by solving (in the least square sense) the overdetermined linear system 
\begin{equation*}
\mathbf{S}_{NM}^{0} \mathbf{c}=\mathbf{S}_\ell^{0}
\end{equation*}
where ${\bf{S}_\ell}^{0}=[S_\ell(0),\, S_\ell(1),\, \dots,\, S_\ell(N-1)]^T$
and  $\mathbf{K}_{NM}^{0}$ is the Casorati matrix associated to the monomial powers $\{k^s z_j^k\}$ for $k=1,...,N-1$.

If $m_j \equiv 1$, the Casorati matrix ${\bf K}^{0}_{Nn}$ reduces to the Vandermonde matrix $(V)_{ij}=z_j^{i+1}$ of order $N \times n$ associated to the  zeros $ z_1,\, \dots,\, z_n$.
The solution of the Casorati system allows us to immediately compute the norming constants as $(\Gamma_\ell)_{js}=s! c_{js}.$

The coefficients $\{ (\Gamma_r)_{js} \}$ are then obtained by solving, in the least square sense, a linear system whose vector of known data is given by $\Omega_r(\alpha)$ evaluated in a sufficiently large set of points $N>M$.

\section{Numerical results} \label{tests}
In order to access the effectiveness of our method in the approximation of $\Omega_\ell$, we adopted the following error estimate
$$E_{r,n}= \frac{\|\tilde{\Omega}_{\ell,n}-\Omega_\ell\|_\infty}{\|\Omega_\ell\|_\infty}$$
where $\tilde{\Omega}_{\ell,n}(y)$ is the computed Marchenko kernel in $n$ equispaced point of $[0,\,2L]$.
Let us now show our results in two cases in which $\Omega_\ell$ is analitically known.

{\emph{Test 1.}}
Let us consider as initial potential the soliton given by
\begin{equation}
u_0(x)= \frac{-2ce^{-2ax}}{1+\frac{|c|^2}{4p^2}e^{-4px}}, \quad x \in \RR,
\end{equation}
where $c$, $a$ and $p$ are real parameters \cite{ZS}.
In this case the Marchenko kernel on the left is
$$\Omega_\ell(x)=e^{-ax}.$$
Considering that $u_0(x) \leq 10^{-13}$ if $|x|>15$, in our computation we assumed $L=15
$. Hence taken $n$ and a steplenght $h$ such that $nh=2L$, we computed $\Omega_\ell(\alpha)$ in the $n$ points $\alpha_i=ih$, $i=0,1,...,n$ and we reported in Table \ref{table} the relative error $E_{r,n}$ for different values of $n$.

\begin{table}
\begin{minipage}{3 cm}
\centering
\begin{tabular}{|l|l|}
\hline
$n$ & $E_{r,n}$ \\
\hline
300 & 1.03e-03 \\
600 & 2.62e-04 \\
900 & 1.17e-04 \\
1200 & 6.60e-05 \\ 
\hline
\end{tabular}  
\end{minipage} \hspace{2cm} 
\begin{minipage}{3 cm}
\centering
\begin{tabular}{|l|l|}
\hline
$n$ & $E_{r,n}$ \\
\hline
300 & 8.42e-03 \\
600 & 2.08e-03 \\
900 & 9.27e-04\\
1200 & 5.21e-04\\ 
\hline
\end{tabular}  
\end{minipage}
\caption{Relative errors $E_{r,n}$ for $\Omega_\ell$ in the soliton case (left) and in the multisoliton case (right) \label{table}}
\end{table}

{\emph{Test 2.}}
Let us now take as  initial potential  a multisoliton represented by four solitons which interact each other nonlinearly \cite{VanDerMee2013}, namely
\begin{equation}
u_0(x)=
\begin{cases}
-2 {\bf{b}}^* [ e^{2x{{\bf{A}}^*}}+ {\bf Q} e^{-2x{\bf{A}}} {\bf N} ]^{-1} {\bf{c}}^*  & x \geq 0 \\
-2 {\bf{c}} [ e^{-2x{{\bf{A}}}}+ {\bf N} e^{2x{\bf{A}}^*} {\bf Q} ]^{-1} {\bf{b}}  & x < 0 
\end{cases}
\end{equation}
where $\bf{b}$ and $\bf{c}$ are column and row vectors, respectively, ${\bf A}$ is a matrix with eigenvalues $\alpha_j$ having positive real parts and ${\bf N}$ and ${\bf Q}$ are two matrices obtained by solving the respective Lyapunov equations:
$${\bf Q}{\bf A}+{\bf A}^* {\bf Q}={\bf c}^* {\bf c}, \quad {\bf A}{\bf N}+{\bf N}{\bf A}^*=\bf{b} \bf{b}^*.$$
In this case, it is possible to prove \cite{VanDerMee2013} that the initial Marchenko kernel $\Omega_\ell$ is given by
$$\Omega_\ell(\alpha)= {\bf{c}} e^{-\alpha{\bf{A}}} {\bf{b}},$$
the reflection coefficients $R(\lambda)=L(\lambda)=0$, the bound state terms are $\lambda_j=i a_j$ and the norming constants $\Gamma_{\ell,j}=b_j c_j$.

As in the previous example we assumed $L=15$ as $u_0(x) \leq 10^{-13}$ for $|x| > 15$. Moreover, for simplicity we considered 
$$ {\bf A}= diag([1, 2, 3, 4]), \quad {\bf b}=[1, \, 2, \, -2, \, -1]^T, \quad {\bf c}=[2, \, 1, \, 1, \, 2],$$
which implies that
$$\Omega_\ell(\alpha)=\sum_{j=1}^4 b_j c_j e^{-\alpha a_j }.$$
The results reported in Table \ref{table} show that, as expected, the relative error $E_{r,n}$ decreases with respect to $n$.

\section{Conclusions and perspectives}\label{conclusion}
The effectiveness of the numerical solution of the direct scattering problem in the NLS, as probably in various other NPDEs of integrable type, basically depends on the effectiveness of the numerical solution of Volterra's systems of integral equations with structured kernels on unbounded domains and then on the identification of parameters in monomial-exponential sums.
Our experiments show that in the reflectionless case our matrix-pencil method for the identification of spectral parameters is fully reliable whenever the relative error coming from the solution of systems of Volterra is small enough.

A challenging mathematical problem is to develop effective algorithms to approximate quite well the reflection coefficients, that is to compute ratios \eqref{L}-\eqref{R} or to generate alternative formula for their evaluation. This challenging task is devoted to another paper, as well the generation of a new family of numerical methods for computing more efficiently the auxiliary functions and the Marchenko kernels.   
\bibliographystyle{amsplain}
\bibliography{biblio}

\providecommand{\bysame}{\leavevmode\hbox to3em{\hrulefill}\thinspace}
\providecommand{\MR}{\relax\ifhmode\unskip\space\fi MR }
\providecommand{\MRhref}[2]{%
  \href{http://www.ams.org/mathscinet-getitem?mr=#1}{#2}
}
\providecommand{\href}[2]{#2}
\begin{thebibliography}{10}

\bibitem{AC}
M.J. Ablowitz and P.A. Clarkson, \emph{Solitons, nonlinear evolution equations
  and inverse scattering}, Cambridge University Press, Cambridge, 1991.

\bibitem{Ablowitz2004}
M.J. Ablowitz, B.~Prinari, and A.D. Trubatch, \emph{Discrete and {C}ontinuous
  {N}onlinear {S}chr\"odinger {S}ystems}, Cambridge University Press,
  Cambridge, 2004.

\bibitem{AS}
M.J. Ablowitz and H.~Segur, \emph{Solitons and the inverse scattering
  transform}, SIAM, Philadelphia, 1981.

\bibitem{Agrawal}
G.~P. Agrawal, \emph{Nonlinear fiber optics}, Academic Press, New York, 2001.

\bibitem{ArRoSe2011}
A.~Aric{\`o}, G.~Rodriguez, and S.~Seatzu, \emph{Numerical solution of the
  nonlinear {S}chr\"odinger equation, starting from the scattering data},
  Calcolo \textbf{48} (2011), no.~1, 75--88.

\bibitem{DM2}
F.~Demontis and C.~Van~der Mee, \emph{Explicit solutions of the cubic matrix
  nonlinear {S}chr\"o\-din\-ger equation}, Inverse Problems (2008), no.~24,
  02520, 16 pp.

\bibitem{DM1}
F.~Demontis and C.~Van~der Mee, \emph{Marchenko equations and norming constants
  of the matrix {Z}akharov-{S}habat system}, Operators and Matrices (2008),
  no.~2, 79--113.

\bibitem{FT}
L.D. Faddeev and L.A. Takhtajan, \emph{Hamiltonian methods in the theory of
  solitons}, Classics in Mathematics, Springer, New York, 1987.

\bibitem{FeVanSe2013}
L.~Fermo, C.~Van~der Mee, and S.~Seatzu, \emph{Parameter estimation of
  monomial-exponential sums}, submitted, arXiv:1310.7095, 2013.

\bibitem{KS}
M.~Klaus and J.K. Shaw, \emph{On the eigenvalues of zakharov-shabat systems},
  SIAM J. Math. Anal. (2003), no.~34, 759--773.

\bibitem{KV}
M.~Klaus and K.~Van~der Mee, \emph{Wave operators for the matrix
  za\-kha\-rov-shabat system}, J. Math. Phys. (2010), no.~51, 053503, 26 pp.

\bibitem{Mn74}
S.V. Manakov, \emph{On the theory of two-dimensional stationary self-focusing
  of electromagnetic waves}, Sov. Phys. JETP (1974), no.~38 also: Zh. Eksp.
  Teor. Fiz. {\bf 65}, 505--516 (1973) [Russian], 248--253.

\bibitem{NMPZ}
S.P. Novikov, S.V. Manakov, L.B. Pitaevskii, and V.E. Zakharov, \emph{Theory of
  solitons. {T}he inverse scattering method}, Plenum Press, New York, 1984.

\bibitem{Shaw}
J.~K. Shaw, \emph{Mathematical principles of optical fiber communications},
  CBMS-NSF Regional Conference Series 76 SIAM Philadelphia, 2004.

\bibitem{VanDerMee2013}
C.~Van~der Mee, \emph{Nonlinear evolution models of integrable type}, 11, SIMAI
  e-Lecture Notes, 2013.

\bibitem{ZS}
V.E. Zakharov and A.B. Shabat, \emph{Exact theory of two-dimensional
  self-focusing and one dimensional self-modulation of waves in nonlinear
  media}, Sov. Phys. JETP (1972), no.~34 also: Zh. Eksp. Teor. Fiz. {\bf 61},
  118--134 (1971) [Russian]., 62--69.

\end{thebibliography}

\end{document}


@author:
@affiliation:
@title:
@language: English
@pages:
@classification1:
@classification2:
@keywords:
@abstract:
@filename:
@EOI

